\documentclass[10pt,a4paper,reqno]{amsart}
\usepackage[utf8]{inputenc}

\usepackage{amsmath}
\usepackage{amsthm}
\usepackage{amssymb}
\usepackage{graphicx}
\usepackage{hyperref}
\usepackage{overpic}

\usepackage{enumerate}

\newcommand{\R}{\ensuremath{\mathbb{R}}}

\DeclareMathOperator{\ee}{e}

\newtheorem {theorem} {Theorem}%[section]
\newtheorem {definition}[theorem] {Definition}
\newtheorem {proposition}[theorem] {Proposition}
\newtheorem {corollary}[theorem] {Corollary}
\newtheorem {lemma}[theorem] {Lemma}

\newtheorem {remark}[theorem] {Remark}

\title[Weak-Coppel problem for a class of Riccati differential equations]{Weak-Coppel problem for a class of\\ Riccati differential equations}

\author[A. Gasull]{Armengol Gasull}
\address{Departament de Matem\`{a}tiques, Universitat Aut\`{o}noma de Barcelona, 08193 Be\-lla\-ter\-ra, Barcelona (Spain)}
\email{armengol.gasull@uab.cat}

\author[D. D. Novaes]{Douglas D. Novaes}
\address{Departamento de Matem\'{a}tica, Instituto de Matem\'{a}tica, Estatística e Computa\c{c}\~{a}o Cient\'{i}fica (IMECC), Universidade
Estadual de Campinas (UNICAMP), Rua S\'{e}rgio Buarque de Holanda, 651, Cidade Universit\'{a}ria Zeferino Vaz, 13083--859, Campinas, SP, Brazil.} \email{ddnovaes@unicamp.br} 

\author[J. Torregrosa]{Joan Torregrosa$^*$}
\address{Departament de Matem\`{a}tiques, Universitat Aut\`{o}noma de Barcelona, 08193 Be\-lla\-ter\-ra, Barcelona (Spain); Centre de Recerca Matem\`{a}tica, Campus de Be\-lla\-ter\-ra, 08193 Bellaterra, Barcelona (Spain)}
\email{joan.torregrosa@uab.cat}

\makeatletter 
\@namedef{subjclassname@2020}{%
	\textup{2020} Mathematics Subject Classification}
\makeatother
\thanks{$^*$ Corresponding author.}



\allowdisplaybreaks

\begin{document}

\begin{abstract}
We study the $T$-periodic solutions of the real Riccati differential equation $x' = x^2 + \gamma(t),$
where $x=x(t)$ and $\gamma$ is a $T$-periodic function. Our goal is to define a real-valued discriminant $\Delta_{\gamma}$ that determines whether the equation admits two, one, or no $T$-periodic solutions, in analogy with the classical discriminant of the quadratic algebraic equations. This problem is closely related to a question posed by Coppel concerning the characterization of bifurcation curves for planar quadratic differential systems. Although our result does not cover all periodic Riccati differential equations, many of them can be transformed  into this particular form.
\end{abstract}

\maketitle

\section{Introduction}

The existence of $T$-periodic solutions for first-order linear differential equations with $T$-periodic coefficients,
\[
x' = a_1(t)x + a_0(t),
\]
where $a_0$ and $a_1$ are smooth real-valued $T$-periodic functions, is a classical and well-understood problem. It is closely connected to Floquet theory, and it is well known that such equations admit either a unique hyperbolic $T$-periodic solution or a continuum of $T$-periodic solutions; see, for instance, \cite{Inc1944} or any basic text on differential equations. Moreover, it is easy  to elucidate  in terms of $a_0$ and $a_1$ which is the situation.

The problem becomes more difficult  in the case of Riccati differential equations, which take the general form
\begin{equation}\label{riccati}
	x' = a_2(t)x^2 + a_1(t)x + a_0(t),
\end{equation}
with all coefficients $a_i(t)$ again assumed to be smooth real-valued $T$-periodic functions. It is also well known that equation \eqref{riccati} can have either a continuum of $T$-periodic solutions or at most two   $T$-periodic solutions, see~\cite{Hil,Inc1944} and the survey \cite{Gas}. In the latter case, the possibilities are restricted to the following: two hyperbolic periodic solutions, a single periodic  solution (either hyperbolic or double), or no periodic solutions. This classification can be established, for instance, by noting that if $x=\phi(t)$ is a periodic solution, then the Riccati equation can be reduced into a linear equation explicitly depending on $\phi(t)$ through the change of variables $y = 1/(x-\phi(t))$. The classification then follows from the corresponding results for the linear case.

For brevity, and by analogy with the planar autonomous case, isolated periodic solutions of a $T$-periodic differential equation $x’(t)=f(t,x)$ will also be referred to as {\it limit cycles}. In addition, a limit cycle $x=\phi(t)$ is said to be hyperbolic if the corresponding zero of the displacement map (i.e., the Poincaré map at time $T$ minus the identity) is simple. This condition can be expressed as
\begin{equation}\label{eq:hip}
h(\phi)=\int_0^T \frac{\partial f}{\partial x}(t,\phi(t)) \, {\rm d} t \ne 0,
\end{equation}
see~\cite{Llo}. Furthermore, if $h(\phi)<0$ (resp. $h(\phi)>0$), the limit cycle is attractive (resp. repulsive). Similarly, for $f$ smooth enough, a limit cycle is said to have multiplicity $n$ if it corresponds to a zero of multiplicity $n$ of the displacement map. In particular, when $n=2$, it is called a {\it double} limit cycle, in which case it is semi-stable.

 The goal of this paper is to provide a mechanism for precisely determining the number of $T$-periodic solutions within certain subfamilies of~\eqref{riccati}. For more general properties of Riccati equations, we refer the reader to \cite{Gas,Hil,Red1972} and the references therein. In particular, several of these papers relate Riccati equations with the celebrated Hill's equation and also study their extensions to the complex and even to the quaternions.

More specifically, we will focus on the real Riccati equation
\begin{equation}\label{eq:firstnf}
	x'= x^2 + \gamma(t),
\end{equation}
where $\gamma(t)$ is a given smooth real-valued $T$-periodic function. Our objective is to introduce a real-valued discriminant function $\Delta_{\gamma}$, in analogy with the classical discriminant $\Delta = b^2 - 4ac$ associated with quadratic equations $ax^2 + bx + c = 0$. The sign of $\Delta_{\gamma}$ will determine the number of $T$-periodic solutions of \eqref{eq:firstnf}, namely: two periodic solutions when $\Delta_{\gamma} > 0$, one (double) when $\Delta_{\gamma} = 0$, and none when $\Delta_{\gamma} < 0$.

This problem is related to the one proposed by Coppel in \cite{Coppel66} for planar quadratic systems, intended to characterize their phase portraits through algebraic inequalities on the coefficients. However, as shown in \cite{DF91}, such a characterization cannot, in general, be achieved using analytic, much less algebraic, inequalities on the coefficients. This has led to weaker formulations of Coppel's problem that relax the regularity conditions on the bifurcation curves. For instance, in \cite{GGPT,GGT10}, the authors attempted to approximate bifurcation curves by sandwiching them between algebraic curves. In our case, we refer to the Weak-Coppel Problem, since, as expected, our discriminant $\Delta_{\gamma}$ does not appear to be algebraic in $\gamma$.

It is worth mentioning that equation~\eqref{eq:firstnf} encompasses a broader class of Riccati differential equations inside~\eqref{riccati}, namely those for which $a_2(t) \neq 0$ for all $t \in [0,T]$ (see Corollary~\ref{coro}) as well as those introduced in Section~\ref{sec:mes}.

In order to establish the discriminant $\Delta_{\gamma}$, we first introduce the average and the mean-centered function of $\gamma$, given by \begin{equation}\label{average}
\overline{\gamma} := \dfrac{1}{T}\int_0^T \gamma(t) \, {\rm d} t\qquad\mbox{and}\qquad \widehat{\gamma}(t) := \gamma(t) - \overline{\gamma},
\end{equation}
respectively. Next, for each $\gamma\in \mathcal{C}^1_T(\R)$, we define the functionals $\underline{\mu}[\gamma]: \mathcal{C}^1_T(\mathbb{R}) \to \mathbb{R}$ and $\overline{\mu}[\gamma]: \mathcal{C}^1_{T,0}(\mathbb{R}) \to \mathbb{R}$ by
\begin{equation}\label{def:func}
\begin{aligned}
\underline{\mu}[\gamma](p):=&-\underset{t\in[0,T]}{\max}\left\{p(t)^2+\widehat{\gamma}(t)-p'(t)\right\},\\
\overline{\mu}[\gamma](p):=&\left(\displaystyle\int_0^T \ee^{-2\int_0^t p(s)\, {\rm d} s} \, {\rm d} t\right)^{-1}\int_0^T \ee^{-2\int_0^t p(s)\, {\rm d} s}\big(p(t)^2-\widehat{\gamma}(t)\big)\, {\rm d} t.
\end{aligned}
\end{equation}
For $k\in\{0,1\},$ here $\mathcal{C}^k_T(\R)$ denotes the space of $T$-periodic $\mathcal{C}^k$ real functions and $\mathcal{C}^k_{T,0}(\R)\subset \mathcal{C}^k_T(\R)$ denotes the subspace of functions with zero average. For the sake of conciseness, we will denote $\underline{\mu}[\gamma]$ and $\overline{\mu}[\gamma]$ simply as $\underline{\mu}$ and $\overline{\mu}$ when the context is clear.

The following result holds for the above defined functionals. Its proof is a consequence of Proposition~\ref{continuity} and Theorem~\ref{mainaux}.

\begin{proposition}\label{prop:functionals}
Let $\gamma(t)$ be a $T$-periodic function. Then, the functionals $\underline{\mu}[\gamma]$ and $\overline{\mu}[\gamma]$ are continuous and satisfy $\underline{\mu}[\gamma](p)\leq\overline{\mu}[\gamma](q)\leq0$ for any $p\in \mathcal{C}^1_{T}(\R)$ and $q\in \mathcal{C}^1_{T,0}(\R)$. Moreover, there exists $p^*\in \mathcal{C}^1_{T,0}(\R)$ such that 
\begin{equation*}
	\begin{aligned}
\gamma^*:&=\max \left\{\underline{\mu}[\gamma](p):\, p\in  \mathcal{C}^1_{T}(\R)\right\}=\underline{\mu}[\gamma](p^*)\\
&=\min \left\{\overline{\mu}[\gamma](p): p\in  \mathcal{C}^1_{T,0}(\R)\right\}=\overline{\mu}[\gamma](p^*),
	\end{aligned}
\end{equation*}
and $\gamma^*=0$ if and only if $\widehat\gamma(t)\equiv 0$.
\end{proposition}

We are ready to introduce the discriminant of~\eqref{eq:firstnf}.

\begin{definition} Let $\gamma(t)$ be a $T$-periodic function. We define the discriminant $\Delta_{\gamma}$ of its associated Riccati differential equation~\eqref{eq:firstnf},  as 
\begin{equation*}
		\Delta_{\gamma}:=\gamma^*-\overline{\gamma},
	\end{equation*}
where $\gamma^*$ is introduced in Proposition~\ref{prop:functionals} and $\overline{\gamma}$ in~\eqref{average}.
\end{definition}

We can now state our main results.

\begin{theorem}\label{main}
Let $\gamma(t)$ be a $T$-periodic function. If $\Delta_{\gamma} > 0$, then the differential equation \eqref{eq:firstnf},
\[x' = x^2 + \gamma(t),\] 
has exactly two hyperbolic limit cycles; if $\Delta_{\gamma} = 0$, it has a unique semi-stable double limit cycle; and if $\Delta_{\gamma} < 0$, it has no limit cycles. Moreover,  
\begin{equation}\label{ineq}
\underline{\mu}[\gamma](p) - \overline{\gamma} \leq \Delta_{\gamma} \leq \overline{\mu}[\gamma](q) - \overline{\gamma}
\end{equation}
for every $p \in \mathcal{C}^1_T(\mathbb{R})$ and $q \in \mathcal{C}^1_{T,0}(\mathbb{R})$.
\end{theorem}

The proof of Theorem~\ref{main} will follow from Theorem~\ref{mainaux} in Section~\ref{sec:proof}. A straightforward corollary is:

\begin{corollary}\label{coro}
Consider the $T$-periodic Riccati equation~\eqref{riccati},
\begin{equation*}
	x' = a_2(t)x^2 + a_1(t)x + a_0(t),
\end{equation*}
 with $a_2(t)\ne0$ for all $t\in\R.$
Define the $T$-periodic function
\begin{equation}\label{change2}
	\gamma(t) = a_2(t) a_0(t) - A(t)^2 + A'(t),\quad\mbox{where}\quad A(t)= \frac{1}{2} \left(a_1(t) + \frac{a_2'(t)}{a_2(t)}\right).
\end{equation}
Then, by taking the corresponding  $\Delta_{\gamma}$, the same conclusion of Theorem~\ref{main} holds for the Riccati differential equation~\eqref{riccati}.
\end{corollary}

One of the key ideas behind the proof of our results is to adapt certain properties of the (one-parameter) families of {\it rotated vector fields} to our one-dimensional periodic non-autonomous differential equation (see Remark~\ref{rema}). These particular one-parameter families were introduced by Duff in the 1950s~\cite{Duf} to study the movement and bifurcations of limit cycles in planar autonomous vector fields.  In short, the parameter in these families acts by monotonically rotating the vector field associated with the differential equation (see also~\cite{Per} for further details). This geometric property makes the evolution of the periodic solutions considerably more tractable.

Our results, and more concretely, the inequalities~\eqref{ineq}, are particularly useful to find upper and lower bounds of $\Delta_{\gamma}.$ In Section~\ref{sec:4} we collect several examples of application.  Some of them are simply illustrative examples while others improve previous results for particular Riccati equations. 

As we will see, the central idea of our approach in the examples is to employ a variation of the so-called {\it harmonic balance method}. Briefly, this method consists in approaching the periodic solutions of a differential equation by approximating them through truncated Fourier series, which are obtained by solving intricate nonlinear systems of equations. Although the method is often applied on a heuristic basis (\cite{Mic}), which typically works quite well, there also exist theoretical results establishing its validity under additional hypotheses (\cite{GG,Sto,Ura}). Our variation introduces an additional parameter into the differential equation, treating this parameter as an unknown as well, and then searches for periodic solutions with zero average. As we will see, these periodic solutions correspond to double limit cycles of the Riccati equation and provide information about the vanishing of the discriminant. Further details are given in Section~\ref{sec:methodology}.

In any case, the periodic functions obtained through our variation of the harmonic balance method, when applied to the Riccati equations under study, can be used as admissible choices for $p$ and $q$  in~\eqref{ineq}, thereby yielding bounds for $\Delta_\gamma.$ These bounds are always valid. Moreover, when the harmonic balance method provides a succession of trigonometric polynomials converging to the double limit cycle, the results become considerably stronger. In fact, by Proposition~\ref{prop:functionals}, such approximations provide sharp estimates for $\Delta_\gamma.$

Apart from the results established in Theorem~\ref{main} and Corollary~\ref{coro}, there are a few additional situations in which the number of periodic solutions can be exactly determined. Specifically:  
\begin{itemize}
	\item When an explicit particular solution of~\eqref{riccati} is known; see, for instance,~\cite{CimGasMan} for further details.  
	\item When the Riccati differential equation~\eqref{riccati} can be transformed into~\eqref{eq:firstnf}. This occurs in the case where there exists $u\in\R$ such that $a_2(t)u^2 + a_1(t)u + a_0(t) \ne0,$ for all $t\in\R,$  see Section~\ref{sec:mes}.
\end{itemize}

It is worth noting that there exists a configuration of limit cycles that does not occur in equation~\eqref{eq:firstnf} or in equation~\eqref{riccati} when $a_2$ does not vanish. This configuration corresponds to the case where the Riccati equation admits exactly one hyperbolic limit cycle. To conclude this introduction, we present an illustrative example of this situation. Consider
	\begin{equation}\label{eq:uve}
	x'= \big(v(t)-v'(t)\big) x^2-x,
	\end{equation}
	where $v(t)$ is any smooth real-valued $T$-periodic function. Some computations show that all solutions are given by $x=0$ and $x=\big(v(t)+k \, {\rm e}^{t}\big)^{-1}$ for $k\in\R$. In particular, the only possible periodic solutions are $x=0$, which is hyperbolic and attractive since $h(0)=-T$ (see \eqref{eq:hip}), and $x=1/v(t)$, which is periodic if and only if $v(t)\neq 0$ for all $t\in\R$. In this case, it is hyperbolic and repelling, as $h(1/v(t))=T$. Notice that if $v(t)$ vanishes, then the corresponding coefficient $a_2(t)=v(t)-v’(t)$ of the Riccati equation~\eqref{eq:uve} also vanishes. Indeed, this happens in two possible ways. If $v$ has a multiple zero at some $t_0$, then $v(t_0)=v'(t_0)=0$, which directly implies $a_2(t_0)=0$. Otherwise, all zeros of $v$ are simple. In this case, taking two consecutive zeros $0\leq t_0<t_1<T$ (which exist by the $T$-periodicity of $v$), one has $v’(t_0)v’(t_1)<0$, and consequently $a_2(t_0)a_2(t_1)<0$. By Bolzano’s Theorem, this implies that $a_2$ must vanish at some point. Therefore, every Riccati equation~\eqref{eq:uve} with vanishing $v$ provides an example with a unique hyperbolic limit cycle ($x=0$) and with $a_2(t)$ also vanishing.

\section{Preliminary results}

In this section, we present some preliminary results that will support the developments in the subsequent sections.

\begin{proposition}\label{continuity}
Let $\gamma(t)$ be a $T$-periodic function. Then, the functionals $\underline{\mu}[\gamma]$ and $\overline{\mu}[\gamma]$, defined in \eqref{def:func}, are continuous.
\end{proposition}
\begin{proof}
The functional $\overline{\mu}[\gamma]$ is the composition of continuous functionals. Therefore it is itself continuous.

To show that $\underline{\mu}[\gamma]$ is continuous on $\mathcal{C}^1_T(\R)$, it is sufficient to show the continuity of the functional 
\[
M(p)=\max_{t\in[0,T]}\{p(t)\}, \, p\in \mathcal{C}^0_T(\R).
\]
First, notice that $M(p+q)\leq M(p)+M(q)$, for every $p,q\in \mathcal{C}^0_T(\R)$. Thus, $M(p)=M(q+p-q)\leq M(q)+M(p-q)$, which implies that $M(p)-M(q)\leq M(p-q)$. Interchanging the roles of $p$ and $q$ we also get $M(q)-M(p)\leq M(q-p)$. Hence, taking into account that $M(p)\leq M(q)$, provided that $p(t)\leq q(t)$ for $t\in \R$, we conclude that
\[
|M(p)-M(q)|\leq\max\{M(p-q),M(q-p)\}\leq M(|p-q|)\leq ||p-q||_{\mathcal{C}^0}.
\]
This provides the continuity of the functional $M$ and, consequently, the continuity of~$\underline{\mu}[\gamma]$.
\end{proof}

\begin{lemma}\label{persol}
A periodic solution of the differential equation \eqref{eq:firstnf} is hyperbolic if and only if its average is nonzero. Moreover, to be hyperbolic is equivalent to the existence of a second, distinct periodic solution of \eqref{eq:firstnf}, which is also hyperbolic.
\end{lemma}
\begin{proof}
Let $g(t,x)=x^2+\gamma(t)$. The first part of the lemma follows directly from the fact that a periodic solution $\phi(t)$ of the differential equation \eqref{eq:firstnf} is hyperbolic if and only if
\[
0 \neq \int_0^T \dfrac{\partial g}{\partial x}(t, \phi(t)) \, {\rm d} t = 2 \int_0^T \phi(t) \, {\rm d} t = 2T\overline{\phi},
\]
where we have used~\eqref{eq:hip}.

Now, assume that $\phi(t)$ is a periodic solution of \eqref{eq:firstnf}. By considering the change of variables
\[
y = \dfrac{1}{x - \phi(t)},
\]
the differential equation \eqref{eq:firstnf} is transformed into the linear differential equation
\begin{equation}\label{eq:linear}
y' = -2\phi(t) y - 1,
\end{equation}
which has the following general solution
\[
y(t, y_0) = \ee^{-2\int_0^t \phi(s) \, {\rm d} s} \left( y_0 - \int_0^t \ee^{2\int_0^s \phi(\tau) \, {\rm d} \tau} \, {\rm d} s \right).
\]

If $\phi(t)$ is hyperbolic, then, by the first part of the statement, we have $\overline{\phi}\neq0$. Consequently, the map
\begin{equation}\label{map}
y_0 \mapsto y(T, y_0) = \ee^{-2T\overline{\phi}} \left( y_0 - \int_0^T \ee^{2\int_0^s \phi(\tau) \, {\rm d}\tau} \, {\rm d} s \right)
\end{equation}
admits a fixed point
\[
\overline{y}_0 = \dfrac{\int_0^T \ee^{-2\int_0^t \phi(s) \, {\rm d}  s} \, {\rm d} t}{1 - \ee^{2T\overline{\phi}}} < 0.
\]
Thus, $y(t, \overline{y}_0)$ is a periodic solution of \eqref{eq:linear} that satisfies $y(t, \overline{y}_0) < 0$ for all $t \in [0, T]$. Reversing the change of variables, we find that
\[
\psi(t) =  \phi(t) +\dfrac{1}{y(t, \overline{y}_0)}
\]
is another periodic solution of the differential equation \eqref{eq:firstnf}. By similar computations, it is not difficult to prove that it is also hyperbolic.

Conversely, if $\phi(t)$ is not hyperbolic, then its average vanishes. Consequently, the map \eqref{map} takes the form  
\[
y_0 \mapsto y(T, y_0) = y_0 - \int_0^T \ee^{-2\int_0^s \phi(\tau) \, {\rm d} \tau} \, {\rm d} s < y_0,
\]  
which has no fixed points. Therefore, the differential equation \eqref{eq:linear} does not admit periodic solutions. This, in turn, implies that the differential equation \eqref{eq:firstnf} has no periodic solutions other than $\phi(t)$.
\end{proof}

In the next corollary we state some direct consequences of Lemma~\ref{persol} and prove that the non-hyperbolic limit cycles of equation \eqref{eq:firstnf} must be double.

\begin{corollary}\label{conseq} Consider the differential equation \eqref{eq:firstnf}.
\begin{itemize}
\item[(a)] If it has two distinct periodic solutions, then they are both hyperbolic. Also, if it has a periodic solution with non-zero average, then it has two distinct hyperbolic periodic solutions.

\item[(b)] If it has a periodic solution with zero average, then it has a unique periodic solution which is non-hyperbolic and double. Also, it has a non-hyperbolic periodic solution if and only  if it has a unique periodic solution.
\end{itemize}
\end{corollary}

\begin{proof}
	All properties are straightforward consequences of Lemma~\ref{persol}, except the character of the non-hyperbolic limit cycles. In~\cite{Llo} it is proved that if  $x=\phi(t)$ is a non-hyperbolic $T$-periodic solution of a $T$-periodic differential equation  $x'(t)=f(t,x)$ and
$\frac{\partial^2 f}{\partial x^2} (t,\phi(t))$
does not vanish, then this periodic solution corresponds to a double (and hence semi-stable) limit cycle.  For the Riccati equation~\eqref{eq:firstnf}, we have $\frac{\partial^2 f}{\partial x^2} (t,\phi(t))\equiv 2,$ so the result follows.	
\end{proof}

We remark that equivalent formulations of the above two results are already known in the literature (see, for instance,~\cite{Gas}). We present them here, together with their proofs, for the sake of completeness and because they will be used in what follows.

\begin{lemma}\label{lem:mu}
Assume that the differential equation \eqref{eq:firstnf} has a periodic solution $\phi(t)$. Then,
\begin{equation}\label{mu}
\overline{\gamma}=-\dfrac{1}{T}\int_0^T \phi(t)^2 \, {\rm d} t\leq0.
\end{equation}
Moreover, $\overline{\gamma}= 0$ if and only if $\gamma(t)\equiv 0$.
\end{lemma}
\begin{proof}
Let $\phi(t)$ be a periodic solution of the differential equation \eqref{eq:firstnf}. Then, $\phi'(t) = \phi(t)^2 + \gamma(t)$. Integrating both sides of this equation  over $t\in[0,T]$, we obtain  
\begin{equation*}
	0=\phi(T)	-\phi(0)=\int_0^T \phi'(t)\, {\rm d} t=\int_0^T \phi(t)^2 \, {\rm d} t+\int_0^T \gamma(t)\,  {\rm d} t= \int_0^T \phi(t)^2 \, {\rm d} t+T\overline{\gamma}.
\end{equation*} 
Thus, the relationship \eqref{mu} follows.  

Now, assuming that $\overline{\gamma}= 0$, it follows from \eqref{mu} that $\phi = 0$. Substituting into the differential equation \eqref{eq:firstnf}, we conclude that $\gamma(t)\equiv 0$.  
\end{proof}

\begin{corollary}\label{nonexis}
Consider the differential equation \eqref{eq:firstnf}. If $\overline{\gamma}>0$, or  $\overline{\gamma}=0$ and $\gamma(t)\not\equiv0$, then it has no periodic solutions. 
\end{corollary}

\section{Proof of the main result}\label{sec:proof}

In order to fully describe the periodic dynamics of the differential equation \eqref{eq:firstnf}, we decompose the function $\gamma(t)$ into the sum $\gamma(t) = \widehat{\gamma}(t) + \overline{\gamma},$ where $\overline{\gamma}$ and $\widehat{\gamma}(t)$ are given by \eqref{average}. We, then, consider the differential equation of type \eqref{eq:firstnf} embedded in the following one-parameter family of differential equations:
\begin{equation}\label{eq:nf}
x' = g(t, x, \mu) := x^2 + \widehat{\gamma}(t) + \mu,\,\, \mu\in\R.
\end{equation}
Notice that when $\mu = \overline{\gamma}$, the differential equation \eqref{eq:nf} becomes \eqref{eq:firstnf}. 

For the differential equation \eqref{eq:nf}, we are able to provide a complete description of the bifurcation diagram as $\mu$ varies.

\begin{theorem}\label{mainaux}
Let $\gamma(t)$ be a $T$-periodic function. Then, there exists a unique $\mu^*\leq0$, satisfying \[
\underline{\mu}[\gamma](p)\leq\mu^*\leq\overline{\mu}[\gamma](q)\leq0
\]
for every $p\in \mathcal{C}^1_{T}(\R)$ and $q\in \mathcal{C}^1_{T,0}(\R)$, for which: if $\mu<\mu^*$ the differential equation \eqref{eq:nf} has exactly two limit cycles, which are hyperbolic; if $\mu=\mu^*$ it has a unique limit cycle $\phi^*(t)$, which is semi-stable and double; and if $\mu>\mu^*$ it has no limit cycles. Moreover,
\begin{equation*}
\begin{aligned}
\mu^*&=\max \left\{\underline{\mu}[\gamma](p):\, p\in  \mathcal{C}^1_{T}(\R)\right\}=\underline{\mu}[\gamma](\phi^*)\\
&=\min \left\{\overline{\mu}[\gamma](p): p\in  \mathcal{C}^1_{T,0}(\R)\right\}=\overline{\mu}[\gamma](\phi^*),
\end{aligned}
\end{equation*}
and $\mu^*=0$ if and only if $\widehat \gamma(t)\equiv 0$, and in this case $\phi^*(t)\equiv 0$.
\end{theorem}

\begin{proof}[Proofs of Proposition~\ref{prop:functionals} and Theorem~\ref{main} ]
Notice that both results follow directly from Theorem~\ref{mainaux} by setting $\gamma^* = \mu^*$ and $\mu = \overline{\gamma}$. 
\end{proof}

\begin{proof}[Proof of Corollary~\ref{coro}] The change of variables  
	\begin{equation*}
		y = a_2(t) x +A(t)=a_2(t) x + \frac{1}{2} \left(a_1(t) + \frac{a_2'(t)}{a_2(t)}\right)
	\end{equation*}
	transforms~\eqref{riccati} into \eqref{eq:firstnf}, that is $
	y' = y^2+ \gamma(t).$ By applying Theorem~\ref{main} to this simplified Riccati differential equation, the proof follows.
\end{proof}

\begin{proof}[Proof of Theorem~\ref{mainaux}]
To investigate the existence of periodic solutions of the differential equation \eqref{eq:nf}, we will repeatedly examine its behavior on specific curves. Let us first briefly outline the general reasoning. 

If one can find a $T$-periodic function $y(t)$ such that
\begin{equation}\label{ine1}
\langle \nabla(x - y(t)), (1, g(t, y(t), \mu)) \rangle \leq 0,\,\, \forall t\in[0,T],
\end{equation}
two possibilities arise: either the inner product vanishes identically, in which case $y(t)$ is a periodic solution of \eqref{eq:nf}, or it does not. In the latter case, since $g(t, x, \mu)$ is bounded in $t$ and the coefficient of $x^2$ is positive, there exist $x_1 < 0 < x_2$ such that $g(t, x_1, \mu) > 0$ and $g(t, x_2, \mu) > 0$ for all $t \in [0,T]$. Therefore, the regions
\[
\mathcal{A} = \{(t,x) : x_1 < x < y(t),\, t \in [0,T]\}\, \text{ and }\,
\mathcal{B}  = \{(t,x) : y(t) < x < x_2,\, t \in [0,T]\}
\]
are positively and negatively invariant under the solutions of \eqref{eq:nf}, respectively. This implies the existence of periodic solutions in both regions. Specifically, since \eqref{eq:nf} has at most two periodic solutions, we conclude that $\mathcal{A}$ contains a unique asymptotically stable periodic solution, while $\mathcal{B}$ contains a unique unstable periodic solution. From item (a) of Corollary~\ref{conseq}, both solutions are hyperbolic. In this case, we say that the curve $y(t)$ separates the periodic solutions.

On the other hand, if the reverse inequality holds in \eqref{ine1}, specifically, \begin{equation}\label{ine2}
	\langle \nabla(x - y(t)), (1, g(t, y(t), \mu)) \rangle \geq 0,\, \forall t\in[0,T],
\end{equation}
two possibilities arise again: either the inner product vanishes identically, in which case $y(t)$ is a periodic solution of the differential equation \eqref{eq:nf}, or it does not. In the latter case, we conclude that the solution of \eqref{eq:nf} with initial condition $y(0)$ is not periodic.

Of course, inequalities~\eqref{ine1} and~\eqref{ine2} do not cover all possible cases, as the product may change sign. However, in such situations, no definitive information can be inferred from $y(t)$, at least in principle.

We now proceed to the proof of the statement.  

First, from Lemma~\ref{lem:mu}, the differential equation \eqref{eq:nf} does not have periodic solutions for any $\mu>0$.

Now, let $\mu < \mu^- := -\max\{\widehat{\gamma}(t) : t \in [0,T]\} \leq 0$. Observe that $g(t,0,\mu) < 0$ for all $t \in [0,T]$. Therefore, following the reasoning outlined at the beginning of this proof, we conclude that the differential equation \eqref{eq:nf} must have two hyperbolic periodic solutions separated by the line $\{(t,0):t\in[0,T]\}$: one asymptotically stable, denoted by $\phi^-(t; \mu)$, and one unstable, denoted by $\phi^+(t; \mu)$.  

Next, we examine how the periodic solutions $\phi^-(t;\mu)$ and $\phi^+(t;\mu)$ evolve as the parameter $\mu$ varies. Let $\mu_1 < \mu_2$, for which the hyperbolic periodic solutions described above exist for $\mu=\mu_1$ and $\mu=\mu_2$. Observe that
\[
\langle \nabla(x - \phi^{\pm}(t,\mu_2)), (1, g(t, \phi^{\pm}(t,\mu_2), \mu_1)) \rangle = \mu_1 - \mu_2 < 0.
\]
The reasoning presented at the beginning of this proof implies that each one of the curves $\phi^-(t, \mu_2)$ and $\phi^+(t, \mu_2)$ separate the periodic solutions $\phi^-(t, \mu_1)$ and $\phi^+(t, \mu_1)$. In other words, the following inequality holds:
\[
\phi^-(t, \mu_1) < \phi^-(t, \mu_2) < \phi^+(t, \mu_2) < \phi^+(t, \mu_1)
\]
for all $t \in [0,T]$. Consequently, the initial condition of the asymptotically stable periodic solution, $x^-(\mu) := \phi^-(0, \mu)$, is strictly increasing with respect to $\mu$, while the initial condition of the unstable periodic solution, $x^+(\mu): = \phi^+(0, \mu)$, is strictly decreasing with respect to $\mu$. Therefore, there exists a $\mu^* \in (\mu^-, 0]$ at which the two hyperbolic periodic solutions coalesce into a double semi-stable periodic solution, which we denote by $\phi^*(t)$.  We claim that this critical parameter $\mu^*$ is unique. Assume, by contradiction, that there exists another parameter $\mu^\circ \in (\mu^*, 0]$ such that the differential equation \eqref{eq:nf} has a non-hyperbolic periodic solution $\phi^\circ(t)$. Since
\[
\langle \nabla(x - \phi^\circ(t)), (1, g(t, \phi^\circ(t), \mu^*)) \rangle = \mu^* - \mu^\circ < 0,
\]
the reasoning outlined at the beginning of this proof implies that, at $\mu = \mu^*,$ besides the double semi-stable periodic solution $\phi^*(t),$ there would exist two additional distinct periodic solutions, leading to a contradiction. Therefore, $\mu^*$ is the unique value of $\mu$ for which the differential equation has a non-hyperbolic periodic solution.

In summary, we have established the existence of a critical parameter $\mu^* \leq 0$, such that: if $\mu < \mu^*$, the differential equation \eqref{eq:nf} has exactly two hyperbolic periodic solutions; if $\mu = \mu^*$, the equation has a unique periodic solution, $\phi^*(t)$, which is semi-stable and double; and if $\mu > \mu^*$, no periodic solutions exist for the differential equation \eqref{eq:nf}.

Now, let $p \in \mathcal{C}^1_T(\mathbb{R})$. Considering the bifurcation established above, to show that $\underline{\mu}(p) \leq \mu^*$, it  is sufficient  to prove that the differential equation \eqref{eq:nf} has at least one periodic solution when $\mu = \underline{\mu}(p)$. Observe that
\[
\langle \nabla(x - p(t)), (1, g(t, p(t), \underline{\mu}(p))) \rangle = \underline{\mu}(p) + p(t)^2 + \widehat{\gamma}(t) - p'(t) \leq 0.
\]
Thus, either the inner product is identically zero, implying that $p(t)$ is a periodic solution of \eqref{eq:nf} for $\mu = \underline{\mu}(p)$, or $p(t)$ separates two hyperbolic solutions of the same differential equation for $\mu = \underline{\mu}(p),$ as outlined at the beginning of the proof. Moreover, let $\phi^*(t)$ be the semi-stable double periodic solution of \eqref{eq:nf} for $\mu = \mu^*$. Then,
\[
\underline{\mu}(\phi^*) = -\underset{t}{\max} \left\{ \phi^*(t)^2 + \widehat{\gamma}(t) - (\phi^*(t)^2 + \widehat{\gamma}(t) + \mu^*) \right\} = \mu^*.
\]
Therefore, $\mu^* = \max \left\{ \underline{\mu}(p) : p \in \mathcal{C}^1_T(\mathbb{R}) \right\}.$

Finally, let $q \in \mathcal{C}^1_{T,0}(\mathbb{R})$. Taking into account the bifurcation established above, to show that $\overline{\mu}(q) \geq \mu^*$, it suffices to prove that the differential equation \eqref{eq:nf} has at most one periodic solution for $\mu = \overline{\mu}(q)$. To this end, consider the linear differential equation  
\begin{equation}\label{linear}
y' = 2q(t) y + \widehat{\gamma}(t) + \mu - q(t)^2,
\end{equation}
which has the general solution  
\[
y(t, y_0) = \ee^{2\int_0^t q(s) \, {\rm d} s} \Bigg(y_0 + \mu \int_0^t \ee^{-2\int_0^s q(r) \, {\rm d} r} \, {\rm d} s + \int_0^t \ee^{-2\int_0^s q(r) \, {\rm d} r} (\widehat\gamma(s) - q(s)^2) \, {\rm d} s \Bigg).
\]
Thus, one can easily verify that for $\mu = \overline{\mu}(q)$, the differential equation \eqref{linear} exhibits a center, meaning that $y(T, y_0) = y_0$ for every $y_0 \in \mathbb{R}$. On the other hand, following the reasoning at the beginning of the proof, we compute  
\[
\langle \nabla(y - y(t, y_0)), (1, g(t, y(t, y_0), \mu)) \rangle = (q(t) - y(t, y_0))^2 \geq 0,
\]
for every $\mu \in \mathbb{R}$ and $y_0 \in \mathbb{R}$. In particular, for $\mu = \overline{\mu}(q)$, either the above expression does not vanish identically for every $y_0 \in \mathbb{R}$, implying that the differential equation \eqref{eq:nf} has no periodic solutions, or there exists some $y_0^*$ for which the expression vanishes identically. In the latter case, $y(t, y_0^*) = q(t)$ is a periodic solution of the differential equation \eqref{eq:nf} with zero average. Therefore, by item (b) of Corollary~\ref{conseq}, it is the unique periodic solution and is semi-stable and double.  Moreover, let $\phi^*$ be the double semi-stable periodic solution of \eqref{eq:nf} for $\mu = \mu^*$. Since $\phi^*$ has zero average, integration by parts gives  
\[
\int_0^T \ee^{-2\int_0^t \phi^*(s)\, {\rm d} s} {\phi^*}'(t) \, {\rm d} t = 2\int_0^T \ee^{-2\int_0^t \phi^*(s) \, {\rm d} s} \phi^*(t)^2 \, {\rm d} t.
\]
Thus, we obtain  
\[
\begin{aligned}
0 &= \int_0^T \ee^{-2\int_0^t \phi^*(s)\, {\rm d} s} \Big(2\phi^*(t)^2 - {\phi^*}'(t)\Big) \, {\rm d} t \\  
&= \int_0^T \ee^{-2\int_0^t \phi^*(s)\, {\rm d} s} \Big(\phi^*(t)^2 - \widehat{\gamma}(t) - \mu^*\Big) \, {\rm d} t \\  
&= \Big(\overline{\mu}(\phi^*) - \mu^*\Big) \int_0^T \ee^{-2\int_0^t \phi^*(s)\, {\rm d} s} \, {\rm d} t.
\end{aligned}
\]
Hence, it follows that $\overline{\mu}(\phi^*) = \mu^*$, and consequently,  
\[
\mu^* = \min \left\{\overline{\mu}(p) : p \in \mathcal{C}^1_{T,0}(\mathbb{R}) \right\}.
\]

To conclude the proof, we notice that Lemma~\ref{lem:mu} implies that $\mu^*=0$ if and only if $\widehat\gamma(t)\equiv 0$. In this case, $\phi^*(t)\equiv 0$.
\end{proof}

\begin{remark}\label{rema}
	As we have already explained in the introduction, several steps of the proof of Theorem~\ref{mainaux} could have been addressed by using the theory of rotated vector fields, see~\cite{Duf,Per}. To apply it, the Riccati differential equation~\eqref{eq:nf} should be written as the equivalent planar autonomous vector field, defined  on the cylinder $ \R\times\mathbb{S}^1,$ with $\mathbb{S}^1\simeq \R/[0,T],$ 
	\begin{align*}
		\frac{dx}{dt} &=P(x,y,\mu)=x^2+\widehat{\gamma}(y)+\mu,\\
		\frac{dy}{dt} &=Q(x,y,\mu)=1.
	\end{align*}
It has the associated vector field $X(x,y,\mu) = (P(x,y,\mu), Q(x,y,\mu))$. The key condition that qualifies it as a family of rotated vector fields (ensuring that $X(x,y,\mu)$ rotates monotonically as $\mu$ varies) is
	\[
	R(x,y,\mu):=Q(x,y,\mu)\frac{\partial P}{\partial \mu}(x,y,\mu)-P(x,y,\mu)\frac{\partial Q}{\partial \mu}(x,y,\mu)\ne0.
	\]
	This condition holds, because $R(x,y,\mu)\equiv1.$ In any case, we have chosen not to rely on this theory and instead provide a self-contained proof, which essentially employs reasoning similar to that used in the general case.
\end{remark}

\section{Estimating the discriminant: methodology and examples}\label{sec:4}

In this section, we illustrate the application of Theorem~\ref{main} and Corollary~\ref{coro} through several examples. We begin with a straightforward one, demonstrating that the computation of the discriminant $\Delta_{\gamma}$ for \eqref{eq:firstnf} is equivalent to determining $\mu^*$ for the one-parameter family \eqref{eq:nf}. Consider the equation
\begin{equation}\label{ex0}
x' = x^2 - 1 + \cos(t) + \dfrac{1}{2}\cos(2t).
\end{equation}
Referring to \eqref{eq:firstnf} and \eqref{average}, we find
\[
\gamma(t) = -1 + \cos(t) + \dfrac{1}{2}\cos(2t), \,\, \overline{\gamma} = -1, \text{ and } \widehat{\gamma}(t) = \cos(t) + \dfrac{1}{2}\cos(2t).
\]
As discussed at the beginning of Section~\ref{sec:proof}, we embed equation \eqref{ex0} into the one-parameter family of differential equations \eqref{eq:nf}, given by
\begin{equation}\label{ex0nf}
x' = g(t, x, \mu) := x^2+\cos(t) + \dfrac{1}{2}\cos(2t) + \mu.
\end{equation}
According to Theorem~\ref{mainaux}, there exists a unique $\mu^*$ for which $x' = g(t, x, \mu^*)$ admits a double semi-stable periodic solution, or equivalently, from Lemma~\ref{persol}, a periodic solution with zero average. Taking into account that $\phi^*(t) := \sin(t)$ is a periodic solution of $x' = g(t, x, -1/2)$, we conclude that $\mu^* = -1/2$. Thus, by Proposition~\ref{prop:functionals}, $\gamma^* = \mu^* = -1/2$, yielding $\Delta_{\gamma} = \gamma^* - \overline{\gamma} = 1/2> 0.$
Hence, Theorem~\ref{main} ensures that the differential equation \eqref{ex0} has exactly two hyperbolic periodic solutions.

Naturally, if $\phi^*(t)$ was not known, determining the value of $\mu^*$ would not be possible in this way. In the next section, we propose a method to estimate the value of $\mu^* $ in cases where $\phi^* (t)$ cannot be explicitly determined.

\subsection{Methodology}\label{sec:methodology}
Consider the one-parameter family of differential equations:
\[
x' = g(t, x, \mu) := x^2 + \alpha(t) + \mu,\,\, \mu\in\R,
\]
where $\alpha$ is a zero average $T$-periodic function. From Theorem~\ref{mainaux}, we know that
 $\underline{\mu}(p)\leq\mu^*\leq\overline{\mu}(q)\leq0$
for every $p\in \mathcal{C}^1_{T}(\R)$ and $q\in \mathcal{C}^1_{T,0}(\R)$. Furthermore, 
\begin{equation*}
\begin{aligned}
\mu^*&=\max \left\{\underline{\mu}(p):\, p\in  \mathcal{C}^1_{T}(\R)\right\}=\underline{\mu}(\phi^*)\\
&=\min \left\{\overline{\mu}(p): p\in  \mathcal{C}^1_{T,0}(\R)\right\}=\overline{\mu}(\phi^*).
\end{aligned}
\end{equation*}
 
Our methodology generates a sequence of functions $p_n \in \mathcal{C}^1_{T,0}(\R)$, which is expected to converge to $\phi^*$. This allows us to estimate $\mu^*$ as  
\begin{equation}\label{estimate}
\underline{\mu}(p_n) \leq \mu^* \leq \overline{\mu}(p_n), \quad \text{ for all } n \in \mathbb{N}.
\end{equation}  
Although convergence is not guaranteed, if the sequence does converge, the continuity of the functionals $\underline{\mu}$ and $\overline{\mu}$ ensures that the estimate \eqref{estimate} can be made arbitrarily accurate. If the sequence does not converge, the estimate remains valid, through it may fail to achieve the desired accuracy.  

Let us now describe the construction of $p_n(t)$. Denote by $\alpha_n(t)$ the Fourier series of $\alpha(t)$ truncated at order $n$. Let $\lambda = (a_1, \ldots, a_n, b_1, \ldots, b_n) \in \R^{2n}$ and define
\[
x_n(t,\lambda) := \sum_{k=1}^n a_k \cos(k t) + \sum_{k=1}^n b_k \sin(k t).
\]
Notice that  $\overline{x}_n=0.$ Consider the equation
\begin{equation}\label{truncated}
\dfrac{\partial x_n}{\partial t}(t,\lambda) - x_n(t,\lambda)^2 - \alpha_n(t) - \mu = 0, \,\, t \in \R.
\end{equation}
For this identity to hold, we need to solve a system with $2n+1$ equations in the variables $(\mu, \lambda) \in \R^{2n+1}$. If there exists $(\mu_n, \lambda_n) \in \R^{2n+1}$ such that \eqref{truncated} is satisfied, we define $p_n(t) = x_n(t, \lambda_n)$. Notice that, in this case, by integrating~\eqref{truncated} between $0$ and $2\pi$ we get that
$0-\pi ||\lambda_n||^2-0-2\pi \mu_n=0,$ and hence $\mu_n=||\lambda_n||^2/2.$ 

 As we anticipate in the introduction, the procedure described above is a variation of the heuristic method of harmonic balance, which does not guarantee the generation of a convergent sequence. In fact, there is no assurance that for each $n \in \mathbb{N}$, a pair $(\mu_n, \lambda_n) \in \R^{2n+1}$ will exist to define $p_n$. However, if the obtained sequence $(p_n)$ converges uniformly to a function $p^*$, then $\mu_n$ converges to $\mu^*$ and $p^*=\phi^*$. 
Indeed, the relationship $p_n'(t) - p_n(t)^2 - \alpha_n(t) - \mu_n = 0$, $t\in\R,$ implies that
\[
\begin{aligned}
\mu_n=&- \dfrac{1}{T}\int_0^{T} p_n(t)^2 \, {\rm d} t \text{ and }\\
 p_n(t)=&p_n(0)+\int_0^t \left( p_n(s)^2+\alpha_n(s)+\mu_n \right)\, {\rm d} s,\,\, t\in\R.
 \end{aligned}
\]
Thus, assuming $p_n\to p^*$ uniformly, the sequence $\mu_n$ converges to
\[
\tilde\mu:=-\dfrac{1}{T}\int_0^{T} p^*(t)^2\, {\rm d} t
\]
and
\[
\begin{aligned}
p^*(t)&=\lim p_n(t)\\
&=\lim \left(p_n(0)+\int_0^t \left( p_n(s)^2+\alpha_n(s)+\mu_n \right)\, {\rm d} s\right)\\
&=p^*(0)+\int_0^t \left(p^*(s)^2+\alpha(s)+\tilde \mu \right)\, {\rm d} s,
\end{aligned}
\]
implying that $p^*(t)$ is a $T$-periodic solution of $x'=g(t,x,\tilde\mu)$. Moreover, since $p_n$ has zero average for each $n \in \mathbb{N}$, the uniform convergence $p_n \to p^*$ implies that $p^*(t)$ also has zero average. Hence, by Lemma~\ref{persol} and Theorem~\ref{mainaux}, we conclude that $p^* = \phi^*$ and $\tilde \mu = \mu^*$.

In conclusion, the reasoning above justifies the procedure as a reasonable approach for obtaining a sequence of functions $p_n$ to analytically estimate $\mu^*$ as in~\eqref{estimate}, even though, as noted, it does not guarantee an estimate of $\mu^*$ with arbitrary accuracy.

When applying the methodology described above to example \eqref{ex0nf}, the procedure concludes in the first step, yielding $p_1(t) = \sin(t) = \phi^*(t)$ and $\mu^* = -1/2$.

\subsection{Example 1}
Consider the following more interesting example:
\begin{equation}\label{ex1}
x' = g(t, x, \mu) := x^2 + \widehat{\gamma}(t) + \mu, \text{ where } \widehat{\gamma}(t) := \dfrac{45\cos^2(t) - 29}{(3\cos(t) - 5)^2},
\end{equation}
is a periodic function with zero average. It can be verified that
\[
\phi^*(t) := \dfrac{6\sin(t)}{5 - 3\cos(t)}
\]
is a periodic solution of $x' = g(t, x, -1)$. Since $\phi^*(t)$ also has zero average, we conclude, similarly to the previous example, that $\mu^* = -1$.

Now, let us apply the methodology developed in the previous section to analyze the differential equation \eqref{ex1} using the Fourier series
\[
\widehat\gamma(t)=\frac{2}{3}\cos(t) + \frac{14}{9}\cos(2t)+ \frac{26}{27}\cos(3t)+ \frac{38}{81}\cos(4t)+ \frac{50}{243}\cos(5t) + \frac{62}{729}\cos(6t)+\cdots,
\]
to illustrate how our approach recovers this bifurcation value.

In the first step, using the Fourier truncation of order $1$, we solve the linear system 
$\{-a_1=0, 3 b_1 - 2=0\},$
which yields
\[
p_1(t) = \dfrac{2}{3} \sin(t).
\]
In the second step, the parameters of $p_2$ are determined solving 
\[
\{a_1+a_1 b_2 - a_2 b_1 =3b_1  - 3a_1 a_2 - 3b_1 b_2- 2= 2 a_2+a_1 b_1= 36 b_2-9 a_1^2 + 9 b_1^2  - 28=0\},
\]
from which we obtain 
\[
p_2(t) =2\omega\sin(t) -(\omega^2 - 7/9)\sin(2t) \approx 1.175 \sin(t) + 0.4326 \sin(2t),
\]
where $\omega$ is the unique real solution of $9\omega^3 + 2\omega - 3=0$, that is, $\omega\approx 0.5874987922$.

In the third step, determining the parameters of $p_3$ requires solving a polynomial system of degree greater than five but with rational coefficients, whose solutions cannot be expressed in terms of radicals. Nevertheless, these parameters can be computed numerically to arbitrary precision, and if necessary, the estimate of their errors can also be provided. Here, we give only a numerical approximation of the results with 4 significant digits,  although the computations were done with $20$ significant digits:
\[
p_3(t) \approx 1.297 \sin(t) + 0.4416 \sin(2t) + 0.1301 \sin(3t).
\]
By continuing the same procedure, we obtain approximations for the subsequent terms of the sequence:
\[
\begin{aligned}
p_4(t) \approx &1.326 \sin(t) + 0.4440 \sin(2t) + 0.1445 \sin(3t) + 0.04473 \sin(4t), \\
p_5(t) \approx &1.332 \sin(t) + 0.4444 \sin(2t) + 0.1475 \sin(3t) + 0.04853 \sin(4t)\\
& + 0.01512 \sin(5t),\\
p_6(t)\approx& 1.333 \sin (t)+0.4444 \sin (2 t)+0.1480 \sin (3 t)+0.04923 \sin (4 t)\\
&+0.01623 \sin (5 t)+0.005096 \sin (6 t).
\end{aligned}
\]
Now, using \eqref{estimate}, we derive the following upper and lower bounds for $\mu^*$, truncated to three digits:
\[
\begin{aligned}
 -3.333 \approx\underline{\mu}(p_1) &\leq \mu^* \leq \overline{\mu}(p_1)\approx  -0.4897, \\
-1.960 \approx\underline{\mu}(p_2) & \leq \mu^* \leq \overline{\mu}(p_2)\approx-0.9594, \\
-1.430\approx\underline{\mu}(p_3) & \leq \mu^* \leq \overline{\mu}(p_3)\approx-0.9953, \\
-1.174\approx\underline{\mu}(p_4) & \leq \mu^* \leq \overline{\mu}(p_4)\approx-0.9995, \\
 -1.067\approx\underline{\mu}(p_5) & \leq \mu^* \leq \overline{\mu}(p_5)\approx-1.000,\\
 -1.025\approx\underline{\mu}(p_6) & \leq \mu^* \leq \overline{\mu}(p_6)\approx-1.000.
\end{aligned}
\]

As discussed earlier, we know that $\mu^* = -1$, which confirms that the method is providing an accurate estimate of $\mu^*$.

\subsection{Example 2}
As the second example, consider the following one-parameter family of differential equations:
\[
x'=g(t,x,\mu):=x^2+\sin(t)+\mu.
\]
In this case, we do not know the value of $\mu^*$ for which $x'=g(t,x,\mu^*)$ has a zero average periodic solution. So, let us apply the above methodology to approximate the value of $\mu^*$. Note that here $\widehat\gamma$ is already written as its Fourier series.

In the first steps, as in the previous example, the parameters can be determined explicitly, obtaining:
\[
\begin{aligned}
p_1(t)&= -\cos(t),\\
p_2(t)&=\omega\cos(t)+\omega^2\sin(2t)/4 \approx 0.1797 \sin (2 t)-0.8477 \cos (t),
\end{aligned}
\]
where $\omega$ is the unique real solution of $\omega^3 + 4 \omega + 4=0$, that is $\omega\approx-0.8477075981.$

In the subsequent steps, again as in the previous example, determining the parameters of $p_n$, $n\geq 3$, requires solving polynomials of degree greater than five, which cannot be expressed in terms of radicals. However, these parameters can still be computed numerically with arbitrary precision because the coefficients are rational numbers. Here, we provide only approximations of the subsequent terms of the sequence. Although the computations were again done with $20$ digits here, fewer are shown:
\[
\begin{aligned}
p_3(t) \approx& -0.8538 \cos (t)+0.1625 \sin (2 t)+0.04624 \cos (3 t),\\
p_4(t) \approx&-0.8535 \cos (t)+0.1628 \sin (2 t)+0.04279 \cos (3 t)-0.01245 \sin (4 t), \\
p_5(t) \approx &-0.8535 \cos (t)+0.1628 \sin (2 t)+0.04281 \cos (3 t)-0.01172 \sin (4 t)\\
&-0.003395 \cos (5 t),\\
p_6(t)\approx&-0.8535 \cos (t)+ 0.1628 \sin (2 t)+0.04280 \cos (3 t)-0.01172 \sin (4 t)\\
&-0.003235 \cos (5 t)+0.0009309 \sin (6 t).
\end{aligned}
\]

Now, considering \eqref{estimate}, we derive the following upper and lower bounds for $\mu^*$:
\[
\begin{aligned}
-1=\underline{\mu}(p_1) & \leq \mu^* \leq \overline{\mu}(p_1)\approx  -0.3489, \\
-0.5368\approx\underline{\mu}(p_2) & \leq \mu^* \leq \overline{\mu}(p_2)\approx-0.3771, \\
 -0.4264\approx\underline{\mu}(p_3) & \leq \mu^* \leq \overline{\mu}(p_3)\approx-0.3784, \\
-0.3988\approx\underline{\mu}(p_4) &  \leq \mu^* \leq \overline{\mu}(p_4)\approx-0.3785, \\
-0.3853\approx\underline{\mu}(p_5) & \leq \mu^* \leq \overline{\mu}(p_5)\approx-0.3785,\\
-0.3801\approx\underline{\mu}(p_6) & \leq \mu^* \leq \overline{\mu}(p_6)\approx-0.3785.
\end{aligned}
\]
Thus, we conclude that $\mu^* \in [-0.3801, -0.3785]$. This estimate is obtained through several numerical steps. However, if desired, it can also be derived analytically by studying with much more effort all the systems of equations involved, and also the associated functionals~\eqref{def:func}.

\subsection{Example 3}
Consider the one-parameter family of differential equations
\begin{equation}\label{exfamily}
	u' = f(t,u;\eta) := (\sin(t) - 2)u^2 + \eta - \cos(t).
\end{equation}
The goal of this section is to apply our approach to determine the values of $\eta$ for which the equation above admits periodic solutions. The case $\eta = 3$ was studied in~\cite{Ni2018}, where it was shown that exactly two periodic solutions exist.

Since $a_2(t)=\sin(t) - 2\ne0,$ we can apply Corollary~\ref{coro}.   Its associated $\gamma$ given in~\eqref{change2} is 
\[
\gamma_{\eta}(t)= (\sin(t)-2)(\eta -\cos(t)) - \frac{3\cos(t)^2}{4(\sin(t)-2)^2} - \frac{\sin(t)}{2(\sin(t)-2)}
\]
and~\eqref{exfamily} is equivalent to 
\begin{equation}\label{ex2}
x' = x^2 + \gamma_\eta(t).
\end{equation}
Referring to~\eqref{average}, we find
\begin{equation}\label{gammaeta}
\overline{\gamma}_{\eta} = \frac{1}{4} - \frac{\sqrt{3}}{6} -2\eta\qquad\mbox{and}\qquad
\begin{aligned}
\widehat{\gamma}_{\eta}(t)=B(t)+\eta \sin(t),
\end{aligned}
\end{equation}
where
\[
\begin{aligned}
B(t)=&\dfrac{1}{24(\sin(t)-2)^2}\Big(6 (3 \sqrt{3}-7)+228 \cos (t)+(48-16 \sqrt{3}) \sin (t)\\
&-2 \sqrt{3} \cos (2 t)-150 \sin (2 t)-36 \cos (3 t)+3 \sin (4 t)\Big).
\end{aligned}
\]

Following the methodology developed in Section~\ref{sec:methodology}, we embed \eqref{ex2} into the one-parameter family of differential equations  
\begin{equation*}
x' = g(t, x, \eta, \mu) := x^2 + \widehat{\gamma}_{\eta}(t) + \mu,
\end{equation*}  
with the corresponding Fourier series of $\widehat{\gamma}_{\eta}(t)$:
\[
\begin{aligned}
&\frac{1}{6}(6\eta + 24 -13\sqrt{3})\sin(t) + 2\cos(t) - \frac{1}{2}\sin(2t)- \frac{1}{3}(75 - 43\sqrt{3})\cos(2t) \\ &- \frac{1}{6}(792 - 457\sqrt{3})\sin(3t)+\frac{1}{3}(1914 - 1105\sqrt{3})\cos(4t)+\cdots.
\end{aligned}
\]
For each $\eta \in \R$, Theorem~\ref{mainaux} ensures the existence of $\mu^*(\eta)$ such that the differential equation $x' = g(t, x, \eta, \mu^*(\eta))$ admits a zero average periodic solution $\phi_{\eta}^*(t)$.  

To determine the values of $\eta$ for which \eqref{ex2} admits periodic solutions, Theorem~\ref{main} requires analyzing the sign of the discriminant $\Delta_{\gamma_{\eta}}$, given by  
\[
\Delta(\eta) := \Delta_{\gamma_{\eta}} = \mu^*(\eta) - \overline{\gamma}_{\eta}.
\]  
To study the sign of $\Delta(\eta)$, we first establish some properties of the function $\mu^*(\eta)$ in the following result.  

\begin{lemma}\label{lem:smoothconcave}
Consider equation \eqref{exfamily}. Then, the function $\mu^*(\eta)$ is differentiable and concave down.
 \end{lemma}
\begin{proof}
First, notice that, for each $\eta\in \R$ and $p\in \mathcal{C}^1_{T,0}(\R)$,
\begin{equation*}
\overline{\mu}[\gamma_{\eta}](p) = L(p) \eta + M(p),
\end{equation*}
where:
\begin{equation}\label{LM}
\begin{aligned}
L(p)&= -\left(\int_0^T \ee^{-2\int_0^t p(s)\, {\rm d} s} \, {\rm d} t\right)^{-1} \int_0^T \ee^{-2\int_0^t p(s)\, {\rm d} s} \sin(t) \, {\rm d} t, \\
M(p)&= \left(\int_0^T \ee^{-2\int_0^t p(s)\, {\rm d} s} \, {\rm d} t\right)^{-1} \int_0^T \ee^{-2\int_0^t p(s)\, {\rm d} s} \big(p(t)^2 - B(t)\big) \, {\rm d} t.
\end{aligned}
\end{equation}
Now, given $\eta_0\in \R$, let $\phi_{\eta_0}^*(t)$ be the zero average periodic solution of the differential equation $x' = g(t, x, \eta_0, \mu^*(\eta_0))$, which has its existence guaranteed by Theorem~\ref{mainaux}. From the same result, we know that $\mu^*(\eta_0)=\overline{\mu}[\gamma_{\eta_0}](\phi_{\eta_0}^*)$ and
\begin{equation}\label{concavity}
\mu^*(\eta)\leq \overline{\mu}[\gamma_{\eta}](\phi_{\eta_0}^*)=L(\phi_{\eta_0}^*) \eta + M(\phi_{\eta_0}^*),\, \eta\in \R.
\end{equation}
The last equality follows from \eqref{def:func} and \eqref{gammaeta}. This means that, for each  $\eta_0\in \R$, the non-vertical straight line $ \eta\in\R\mapsto L(\phi_{\eta_0}^*) \eta + M(\phi_{\eta_0}^*)$ is tangent to $\mu^*(\eta)$ at $\eta=\eta_0$, implying that $\mu^*(\eta)$ is differentiable. Moreover, \eqref{concavity} implies that the graph of $\mu^*(\eta)$ lies below its tangent lines and, therefore, is concave down.
\end{proof}

We are now ready to prove the following result.

\begin{theorem}
Consider equation \eqref{exfamily}. Then, there exists a unique $\eta^*\in \big((3-2\sqrt{3})/24,1\big)$ for which $\Delta(\eta^*)=0$. Moreover, $\Delta(\eta)<0$ for $\eta<\eta^*$ and  $\Delta(\eta)>0$ for $\eta>\eta^*.$
\end{theorem}
\begin{proof}
First, observe that since $\overline{\gamma}_{\eta}$ is linear in $\eta$ and $\mu^*(\eta)$ is concave by Proposition~\ref{lem:smoothconcave}, the function $\Delta(\eta)$ can have at most two zeros.

On one hand, returning to the differential equation \eqref{exfamily}, we observe that for $\eta \geq 1$, $f(t,0,\eta) \geq 0$ for all $t \in \mathbb{R}$. Moreover, since $f(t,u,\eta) < 0$ for sufficiently large $u$ and all $t \in \mathbb{R}$, it follows that both \eqref{exfamily} and \eqref{ex2} admit exactly two hyperbolic periodic solutions for $\eta \geq 1$.  On the other hand, by Corollary~\ref{nonexis}, the differential equation \eqref{ex2}, and consequently \eqref{exfamily}, has no periodic solutions for $\eta \leq (3-2\sqrt{3})/24$, as this condition guarantees that $\overline{\gamma} \geq 0$. Therefore, Theorem~\ref{main} implies that
\begin{equation}\label{Delta-behavior}
\begin{cases}
\Delta(\eta) > 0, \quad \eta \geq 1,\\
\Delta(\eta) < 0, \quad \eta \leq (3-2\sqrt{3})/24.
\end{cases}
\end{equation}

Hence, by the continuity of $\Delta(\eta)$, there exists $\eta^* \in \big((3-2\sqrt{3})/24,1\big)$ such that $\Delta(\eta^*)=0$. Moreover, if $\Delta(\eta)$ had another zero distinct from $\eta^*$, then \eqref{Delta-behavior} would imply at least three zeros, contradicting the earlier conclusion that $\Delta(\eta)$ has at most two. Thus, $\eta^*$ is unique, and taking \eqref{Delta-behavior} into account, we conclude that $\Delta(\eta) < 0$ for $\eta < \eta^*$ and $\Delta(\eta) > 0$ for $\eta > \eta^*$.
\end{proof}

In what follows, we estimate the function $\mu^*(\eta)$ and the value $\eta^*$ using the methodology developed in the previous section.  We start by giving an outline of the approach we are going to use. 

For given $\eta_0 \in \R$ and $n\in\mathbb{N}$, we compute $p_n^{\eta_0}$, which provide an estimate for $\mu^*(\eta)$ for any $\eta$, though the approximation is most accurate when $\eta = \eta_0$. Notice that  
\[
\overline{\mu}(p_n^{\eta_0}) = L_n^{\eta_0} \eta + M_n^{\eta_0},
\]
where  $L_n^{\eta_0}=:L(p_n^{\eta_0})$ and $M_n^{\eta_0}:=M(p_n^{\eta_0})$ (see \eqref{LM}).

Thus,
\begin{equation}\label{upper}
\mu^*(\eta) \leq L_n^{\eta_0} \eta + M_n^{\eta_0}, \quad \forall \eta \in \mathbb{R}.
\end{equation}
On the other hand, let  
\[
K_n^{\eta_0} = \underset{t\in[0,T]}{\min} \big\{-p_n^{\eta_0}(t)^2 - B(t) - \eta_0 \sin(t) + (p_n^{\eta_0})'(t) \big\}.
\]
Then,  
\[
\begin{aligned}
\underline{\mu}(p_n^{\eta_0}) &= -\underset{t\in[0,T]}{\max} \big\{ p_n^{\eta_0}(t)^2 + B(t) + \eta \sin(t) - (p_n^{\eta_0})'(t) \big\} \\
&= \underset{t\in[0,T]}{\min} \big\{-p_n^{\eta_0}(t)^2 - B(t) - \eta_0 \sin(t) + (p_n^{\eta_0})'(t) - (\eta - \eta_0) \sin(t) \big\} \\
&\geq K_n^{\eta_0} + \underset{t\in[0,T]}{\min} \big\{ -(\eta - \eta_0) \sin(t) \big\} \geq K_n^{\eta_0} - |\eta - \eta_0|.
\end{aligned}
\]
Therefore,
\begin{equation}\label{lower}
\mu^*(\eta) \geq K_n^{\eta_0} - |\eta - \eta_0|, \quad \forall \eta \in \mathbb{R}.
\end{equation}
Hence, for a given $n \in \mathbb{N}$, computing the constants $K_n^{\eta_0}$, $L_n^{\eta_0}$, and $M_n^{\eta_0}$ for a collection $\eta_0\in\{\eta_1, \ldots, \eta_\ell\} $, and combining \eqref{upper} and \eqref{lower}, we obtain two continuous piecewise linear curves that bound $\mu^*(\eta)$ as follows 
\begin{equation}\label{estimfamily}
\mu^*_{L}(\eta):=\max_i \{ K_n^{\eta_i} - |\eta - \eta_i| \} \leq \mu^*(\eta) \leq \min_i \{ L_n^{\eta_i} \eta + M_n^{\eta_i} \}=:\mu^*_{U}(\eta).
\end{equation}
Such an estimation is improved by increasing $n$ and $\ell.$

To illustrate the above approach, let $n=3$ and consider the following five equidistant points in $[(3-2\sqrt{3})/24,1]$:
\[
\{\eta_i:=(3-2\sqrt{3})/24+i(21 + 2\sqrt{3})/96:\, i=0,\ldots,4\}.
\]
Notice that $\eta_0=(3-2\sqrt{3})/24$ and $\eta_4=1$. Thus, after straightforward computations, we get
\[
\begin{aligned}
p_3^{\eta_0}\approx&1.352 \sin (t)-0.416 \sin (2 t)+0.189 \sin (3 t)\\
&+0.028 \cos (t)+0.189 \cos (2 t)-0.056 \cos (3 t),\\
p_3^{\eta_1}\approx&1.322 \sin (t)-0.400 \sin (2 t)+0.161 \sin (3 t)\\&-0.162 \cos (t)+0.288 \cos (2 t)-0.123 \cos (3 t),\\
p_3^{\eta_2}\approx&1.285 \sin (t)-0.365 \sin (2 t)+0.111 \sin (3 t)\\&-0.353 \cos (t)+0.380 \cos (2 t)-0.181 \cos (3 t),\\
p_3^{\eta_3}\approx&1.245 \sin (t)-0.312 \sin (2 t)+0.046 \sin (3 t)\\&-0.539\cos (t)+0.460\cos (2 t)-0.222 \cos (3 t),\\
p_3^{\eta_4}\approx&1.202 \sin (t)-0.247\sin (2 t)-0.026 \sin (3 t)\\&-0.718 \cos (t)+0.525 \cos (2 t)-0.244 \cos (3 t).
\end{aligned}
\]
With those functions, we compute $K_3^{\eta_i}$, $L_3^{\eta_i}$, and $M_3^{\eta_i}$ for $i=0,\ldots,4$, which provide the estimations  \eqref{estimfamily} (see Figure~\ref{fig1}). In addition, bounds for $\eta^*$ can be computed as follows. For each $i = 0, \ldots, 4$, let $\underline{\eta}_i$ be the solution of the equation $\gamma_\eta = L_n^{\eta_i} \eta + M_n^{\eta_i}$, and let $\overline{\gamma}_{\underline{\eta}_i}$ be the corresponding image in \eqref{gammaeta}. The pairs $(\underline{\eta}_i,\overline{\gamma}_{\underline{\eta}_i})$ for $i=0,\ldots,4$ are 
\[
(0.479,-0.997), (0.498,-1.035), (0.505, -1.048), (0.496,-1.031), (0.471,-0.981).
\]
Then, the lower bound $\underline{\eta}$ is defined as the value $\underline{\eta}_i$ for which $\overline{\gamma}_{\underline{\eta}_i}$ attains the minimum among the five considered. That is, $\underline{\eta}\approx0.505.$
Analogously, let $\overline{\eta}_i$ be the solution of the equation $\gamma_{\eta}=K_n^{\eta_i} - |\eta - \eta_i|$ and let $\overline{\gamma}_{\overline{\eta}_i}$ be the corresponding image in \eqref{gammaeta}. The pairs 
$(\overline{\eta}_i,\overline{\gamma}_{\overline{\eta}_i})$ for $i=0,\ldots,4$ are 
\[
(1.442,-2.922), (1.219,-2.477), (1.045,-2.128),(0.907,-1.852), (0.934,-1.907).
\]
Then, the upper bound $\overline{\eta}$ is defined as the value $\overline{\eta}_i$ for which $\overline{\gamma}_{\overline{\eta}_i}$ attains the maximum among the five considered. That is, $\overline{\eta}\approx0.907.$
These two values corresponds with the intersections of the blue line with the black and red ones in Figure~\ref{fig1}. So, we get the following estimation
\[
0.505\approx\underline{\eta}\leq\eta^*\leq \overline{\eta}\approx0.907.
\]

\begin{figure}[h]
	\begin{overpic}[width=10.8cm]{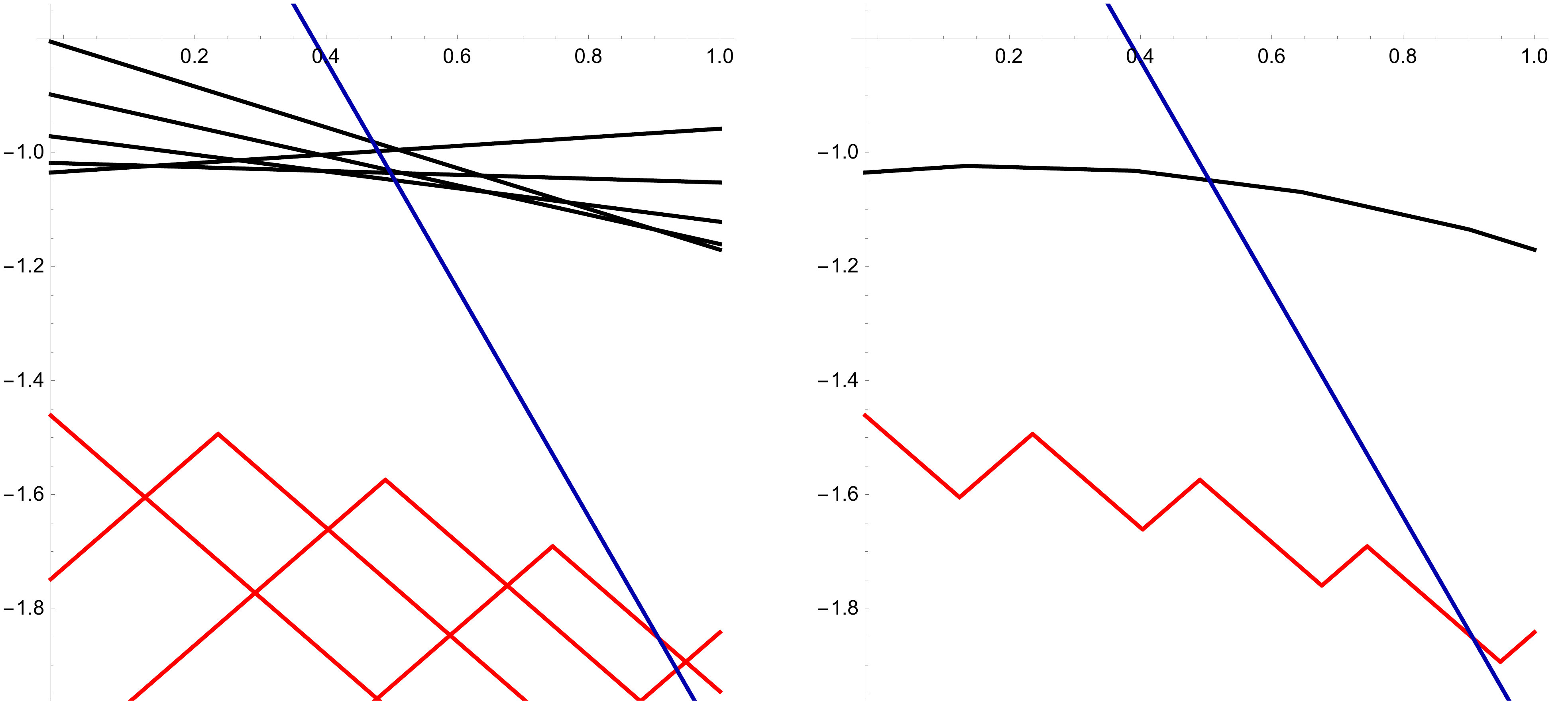}
		\put(33,38){$L_3^{\eta_i} \eta + M_3^{\eta_i}$}
		\put(8,18){$K_3^{\eta_i} - |\eta - \eta_i|$}
		\put(83,35){$\mu^*_{U}(\eta)$}
		\put(58,20){$\mu^*_{L}(\eta)$}
		\put(22,40){$\overline\gamma_{\eta}$}
		\put(74,40){$\overline\gamma_{\eta}$}
	\end{overpic}
	\caption{Left: Set of straight lines $L_3^{\eta_i} \eta + M_3^{\eta_i}$ and piecewise straight lines $K_3^{\eta_i} - |\eta - \eta_i|$, for $i=0,\ldots,4$.
		Right: Continuous functions $\mu^*_{L}(\eta)$ and $\mu^*_{U}(\eta)$.}
	\label{fig1}	
\end{figure}
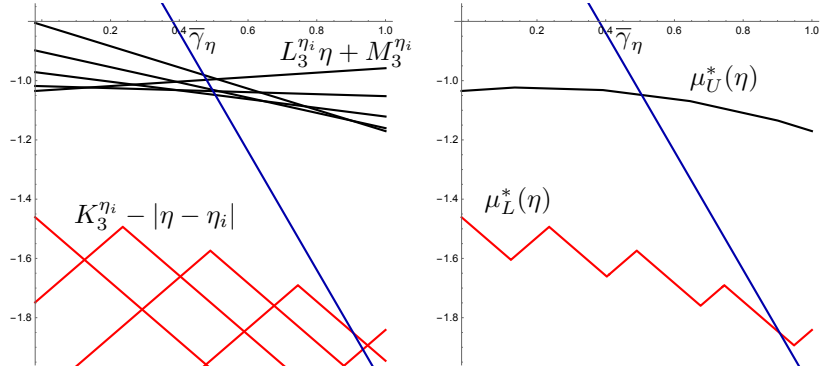

Now, setting $n=8$ and taking $20$ points $\eta_i$ equally distributed in $[(3-2\sqrt{3})/24,1]$ (see Figure~\ref{fig2}), we obtain better approximations:
$0.507\approx\underline{\eta}\leq\eta^*\leq \overline{\eta}\approx0.527.$ 

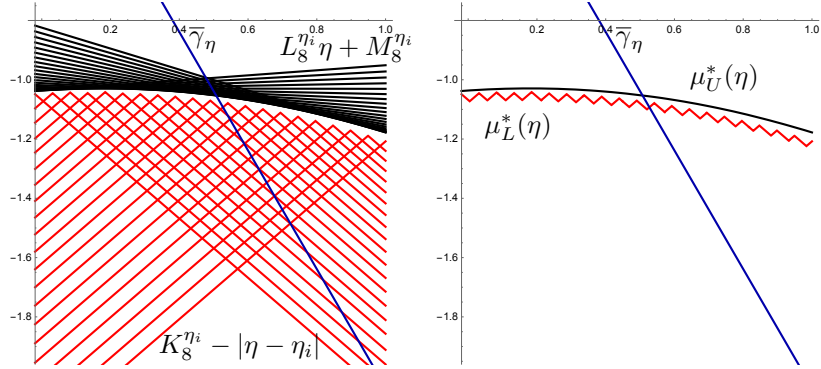
\begin{figure}[h]
	\begin{overpic}[width=10.8cm]{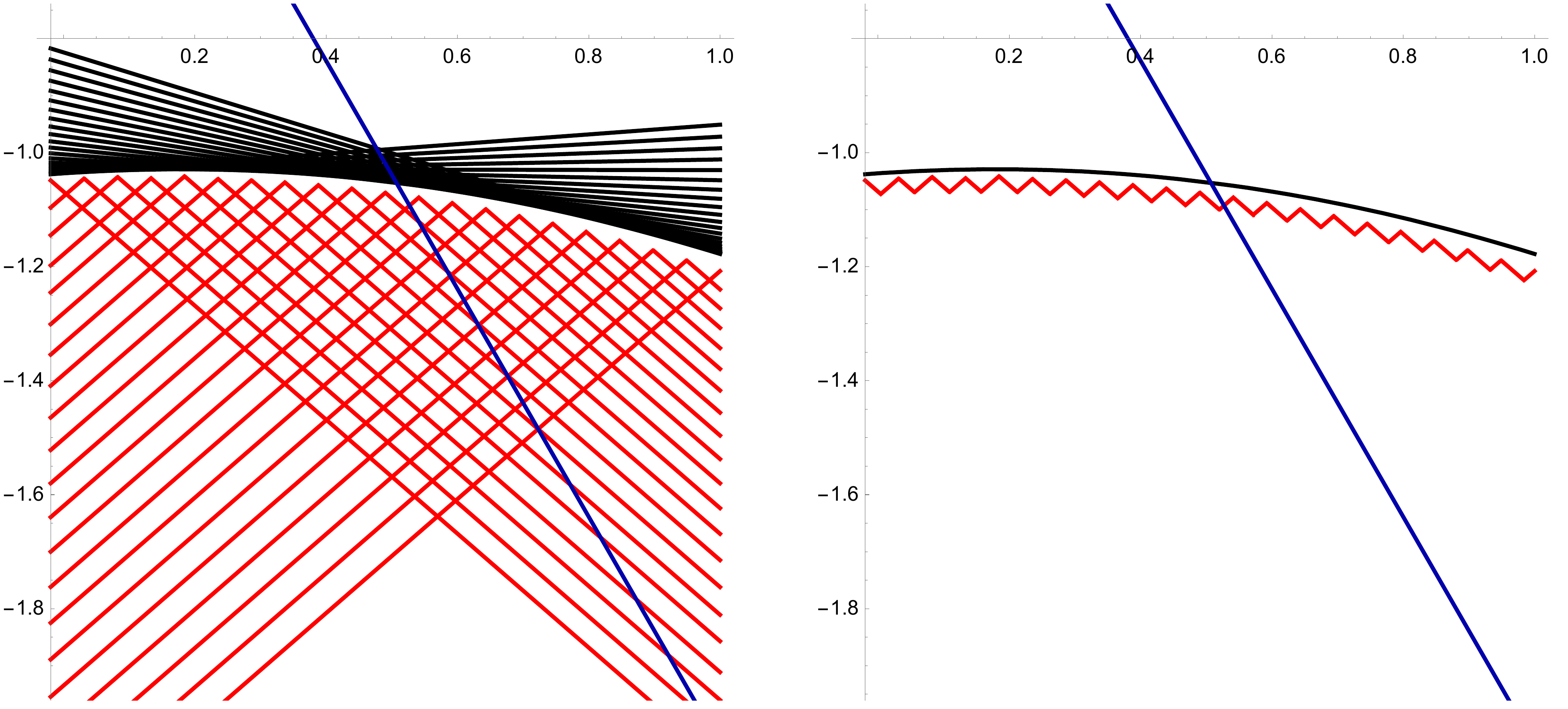}
	\put(33,38.5){$L_8^{\eta_i} \eta + M_8^{\eta_i}$}
	\put(18.3,2){$K_8^{\eta_i} - |\eta - \eta_i|$}
	\put(83,35){$\mu^*_{U}(\eta)$}
	\put(58,29){$\mu^*_{L}(\eta)$}
	\put(22,40){$\overline\gamma_{\eta}$}
	\put(74,40){$\overline\gamma_{\eta}$}
	\end{overpic}
	\caption{Left: Set of straight lines $L_8^{\eta_i} \eta + M_8^{\eta_i}$ and piecewise straight lines $K_8^{\eta_i} - |\eta - \eta_i|$, for $i=0,\ldots,20$.  
Right: Continuous functions $\mu^*_{L}(\eta)$ and $\mu^*_{U}(\eta)$.}
	\label{fig2}	
\end{figure}

Similarly, for $n=15$ with 30 points $\eta_i$ equally distributed in $[(3-2\sqrt{3})/24,1]$ (see Figure~\ref{fig3}), we get  
\[
0.507\approx\underline{\eta}\leq\eta^*\leq \overline{\eta}\approx0.514.
\]
Thus, we obtain a more accurate approximation of the bifurcation value corresponding to the occurrence of a double periodic solution in equation \eqref{exfamily}.

\begin{figure}[h]
	\begin{overpic}[width=10.8cm]{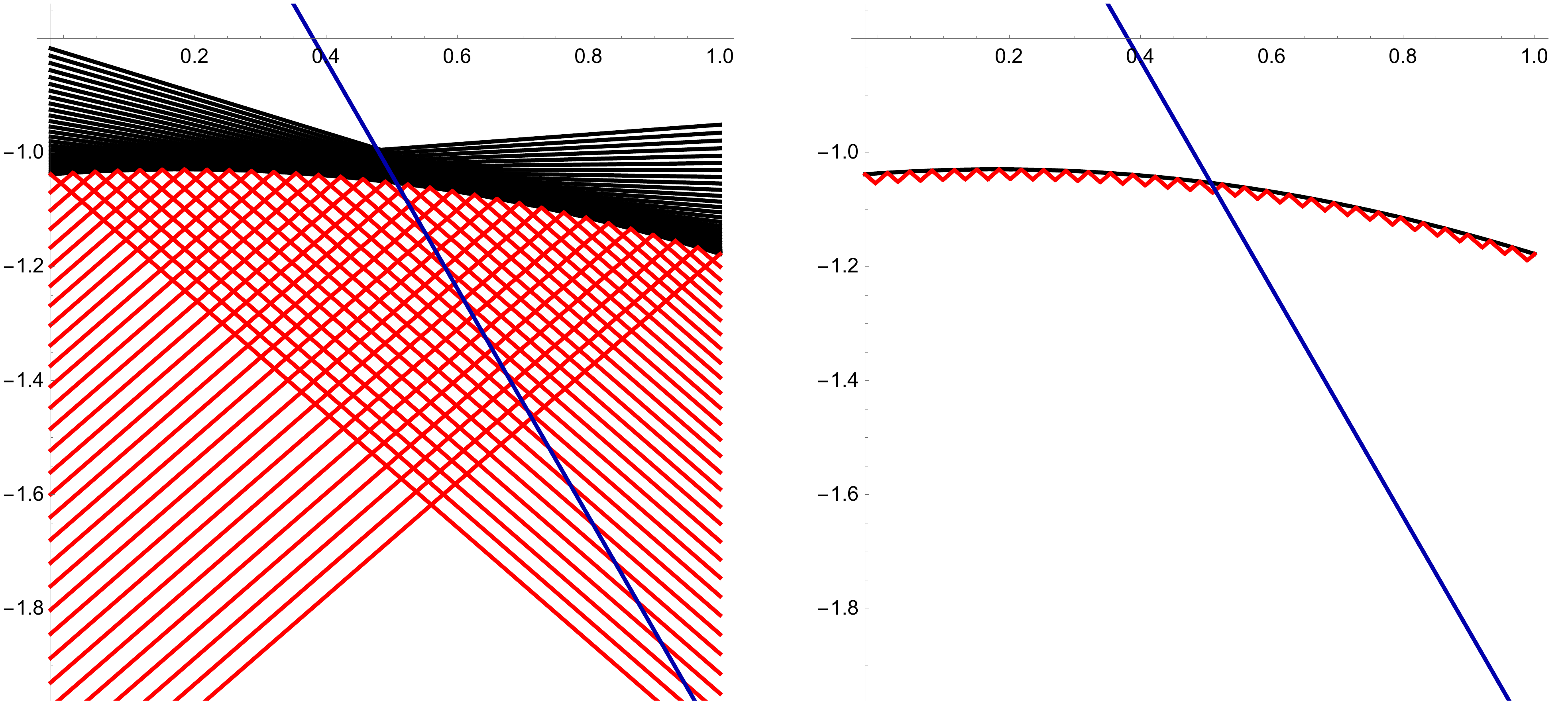}
	\put(33,38.5){$L_{15}^{\eta_i} \eta + M_{15}^{\eta_i}$}
	\put(17.5,2){$K_{15}^{\eta_i} - |\eta - \eta_i|$}
	\put(83,35){$\mu^*_{U}(\eta)$}
	\put(58,30){$\mu^*_{L}(\eta)$}
	\put(22,40){$\overline\gamma_{\eta}$}
	\put(74,40){$\overline\gamma_{\eta}$}
	\end{overpic}
	\caption{Left: Set of straight lines $L_{15}^{\eta_i} \eta + M_{15}^{\eta_i}$ and piecewise straight lines $K_{15}^{\eta_i} - |\eta - \eta_i|$, for $i=0,\ldots,30$.
Right: Continuous functions $\mu^*_{L}(\eta)$ and $\mu^*_{U}(\eta)$.}
	\label{fig3}	
\end{figure}
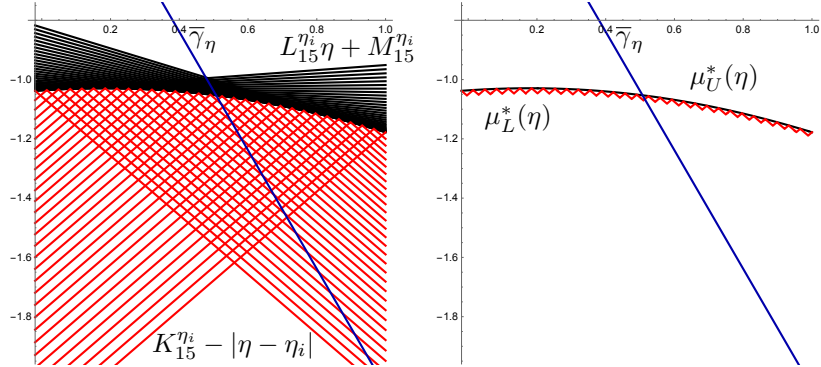

\section{Further comments}\label{sec:mes}
Besides the differential equation \eqref{riccati}, 
\begin{equation*}
	x' = f(t,x)= a_2(t)x^2 + a_1(t)x + a_0(t),
\end{equation*}
which satisfies $a_2(t)\neq0$ for every $t \in [0,T]$, other Riccati differential equations can also be transformed into the canonical form~\eqref{eq:firstnf}. Specifically, if there exists $u_0 \in \mathbb{R}$ such that $f(t, u_0) \neq 0$ for all $t \in [0,T]$, any periodic solution of the differential equation \eqref{riccati} cannot intersect the curve $u = u_0$. Hence, we propose a singular change of variables that maps this curve to infinity.

\begin{lemma}
Consider the differential equation \eqref{riccati}. Suppose there exists $u_0 \in \mathbb{R}$ such that $f(t, u_0) \neq 0$ for all $t \in [0,T]$. Then, the singular change of variables  
\begin{equation}\label{change1}
x =h(t,u):= \frac{m(t)}{u - u_0} + n(t),
\end{equation}
where  
\[
\begin{aligned}
m(t) &= -f(t, u_0), \\  
n(t) &= \frac{1}{2 f(t, u_0)} \left( \frac{\partial f}{\partial t} (t, u_0) - \frac{\partial f}{\partial u} (t, u_0) f(t, u_0) \right),
\end{aligned}
\]
transforms the differential equation \eqref{riccati} into \eqref{eq:firstnf}. Moreover, the solutions of~\eqref{eq:firstnf} that do not intersect the curve $x = n(t)$ correspond bijectively to the solutions of~\eqref{riccati} that do not intersect the curve $u=u_0$ via the inverse transformation
\begin{equation}\label{inversechange1}
u=u_0+\dfrac{m(t)}{x-n(t)}.
\end{equation}
\end{lemma}

The expression for $\gamma$ resulting from the transformation \eqref{change1} can be derived directly, but due to its length and lack of relevance for the main arguments, we omit it here.

Although the change of variables \eqref{change1} is singular, it preserves the complete information regarding the periodic solutions of equation \eqref{riccati}. Specifically, if $\psi(t)$ is a $T$-periodic solution of \eqref{riccati}, then $\psi(t) \neq u_0$ for all $t \in [0,T]$, and thus $\phi(t) = h(t, \psi(t))$ defines a $T$-periodic solution of the transformed equation \eqref{eq:firstnf}. However, the converse does not necessarily hold: not all periodic solutions of \eqref{eq:firstnf} correspond to periodic solutions of \eqref{riccati}. In fact, those trajectories that intersect the curve $x = n(t)$ are not preserved under the inverse transformation \eqref{inversechange1}. Therefore, when studying the periodic solutions of \eqref{riccati} through the transformed system \eqref{eq:firstnf}, it is crucial to account for such intersections, as they may lead to spurious periodic solutions in the transformed setting.

\section*{Acknowledgments}
This work has been realized thanks to the Brazilian S\~{a}o Paulo Research Foundation (FAPESP) grant 2024/15612-6 and Conselho Nacional de Desenvolvimento Cient\'{i}fico e Tecnol\'{o}gico (CNPq) grant 301878/2025-0; the Catalan AGAUR Agency grant 2021 SGR 00113; and the Spanish Ministerio de Ci\'encia, Innovaci\'on y Universidades, Agencia Estatal de Investigaci\'on grants PID2022-136613NB-I00 and CEX2020-001084-M.

\bibliographystyle{abbrv}
\bibliography{biblio}
\end{document}